\documentclass[12pt]{amsart}

\usepackage{amsmath,mathtools}%用于Align环境
\usepackage{amsfonts}%用于\mathbb{}
\usepackage{amsthm}%用于定理环境
\usepackage{amssymb}%第一次用于写右箭头\nrightarrow 命令
\usepackage{enumerate}%计时器，用于编号公式之类;还用于列表环境。
\usepackage{cite}
\usepackage{color}
\usepackage{xcolor}
\usepackage{hyperref}
\usepackage{fancyhdr}

\newtheorem{theorem}{Theorem}[section]

\newtheorem{corollary}[theorem]{Corollary}
\theoremstyle{definition}

\theoremstyle{remark}
\newtheorem{remark}[theorem]{Remark}

\numberwithin{equation}{section}

\newcommand{\ttt}{\mathbb T}

\newcommand{\ccc}{\mathcal{C}}

\newcommand{\bmo}{\operatorname{BMO}}
\newcommand{\vmo}{\operatorname{VMO}}

\allowdisplaybreaks[4]
\begin{document}

% \title[short text for running head]{full title}
\title[Decomposition and characterization of VMO]{Decomposition and characterization of VMO via vanishing Carleson measures}

% Only \author and \address are required; other information is
% optional. Remove any unused author tags.

% author one information
% \author[short version for running head]{name for top of paper}

\author[F. Tao]{Fei Tao \textsuperscript{$a,*$}} 
\address{$a.$ Beijing International Center for Mathematical Research, Peking University, Beijing 100871, P. R. China} 
\email{ferrytau@pku.edu.cn} 
\footnote{*: Corresponding author.}

\author[Y. Yang]{Yaosong Yang \textsuperscript{$b,c$}}
\address{$b.$ Academy of Mathematics and Systems Science, Chinese Academy of Sciences, Beijing 100049, P. R. China}
\address{$c.$ School of Mathematical Sciences, University of Chinese Academy of Sciences, Beijing 100049, P. R. China}
\email{yangyaosong@amss.ac.cn}

\thanks{Research supported by the National Key R \& D Program of China (2025YFA1017500)}

% \subjclass is required.

\date{}

\dedicatory{}
% "Communicated by" -- provide editor's name; required.
\commby{}

% Abstract is required.
\begin{abstract}
% We establish two equivalent characterizations of $\mathrm{VMO}$ in terms of vanishing Carleson measures. 
% The first describes $\mathrm{VMO}$ functions via a decomposition into a continuous boundary term and an integral operator associated with a vanishing Carleson measure. 
% The second characterizes $\mathrm{VMO}$ through the boundary values of smooth functions whose gradients induce vanishing Carleson measures, following an idea from Varopoulos's study of the $\bar{\partial}$-equation.

We establish two equivalent characterizations of $\mathrm{VMO}$ in terms of vanishing Carleson measures. First, we show that any VMO function admits a decomposition into a continuous boundary term and an integral operator associated with a vanishing Carleson measure. Second, motivated by Varopoulos's work on the $\bar{\partial}$-equation, we characterize VMO via the boundary values of smooth functions whose gradients induce vanishing Carleson measures.

As a consequence, we recover the known representation
\[
\mathrm{VMO}=\mathrm{VLO}-\mathrm{VLO},
\]
thereby providing a new perspective on this decomposition.
\end{abstract}

\maketitle

\section{Introduction}
The space $\mathrm{BMO}$, introduced by John and Nirenberg \cite{john1961functions}, is a Banach space modulo constants. The duality between the Hardy space $H^1$ and $\mathrm{BMO}$ leads to the Fefferman--Stein decomposition \cite{stein1972hp}: for any $f\in \mathrm{BMO}(\mathbb{R}^n)$, there exist functions $f_j \in L^{\infty}(\mathbb{R}^n)$ such that \[f(x)=f_0(x)+\sum_{j=1}^n R_j(f_j)(x),\] where $R_j$ denotes the $j$-th Riesz transform. Carleson \cite{carleson1976two} later posed the question of finding a constructive proof of this decomposition. For $n=1$, Jones \cite{jones1980carleson} provided such a construction using complex function theory, while subsequently Uchiyama \cite{Uchiyama82Acta} gave an explicit construction in the general case. 

Another viewpoint arises from the connection with Carleson measures. Carleson \cite{carleson1976two} showed that if $f \in \mathrm{BMO}(\mathbb{R}^n)$ has compact support, then it admits a representation
\[f(x)=f_0(x)+\int_{\mathbb{R}^n \times(0, \infty)} P_t(x-y)\, d \mu(y,t),\] where $f_0 \in L^{\infty}(\mathbb{R}^n)$, $\mu$ is a Carleson measure, and $P_t$ denotes the Poisson kernel. 
This characterizes BMO functions via the balayage of Carleson measures (see the definition below).

In the study of $\bar{\partial} $-equation associated with the Corona problem, Varopoulos \cite{varopoulos1977bmo,varopoulos1978remark} showed that for every 
$f\in \bmo(\mathbb{R}^n)$ there exists a smooth function $F$ on 
$\mathbb{R}^{n+1}_{+}$ such that 
$|\nabla F(x,y)|\,dx\,dy \in CM(\mathbb{R}^{n+1}_{+})$
and the boundary values of $F$ differ from $f$ by a bounded function.

It is natural to ask whether similar decompositions and representations hold for $\mathrm{VMO}$, the closed subspace of $\mathrm{BMO}$ introduced by Sarason \cite{sarason1975functions}. 
Sarason established a vanishing analogue of the Fefferman--Stein decomposition, relating $\mathrm{VMO}$ to the space of bounded uniformly continuous functions $\mathrm{BUC}(\mathbb{R})$. 
However, a direct characterization of $\mathrm{VMO}$ via the balayage of vanishing Carleson measures, or via Varopoulos' approach, does not seem to be available in the literature. The aim of this paper is to provide such characterizations of $\mathrm{VMO}$ in terms of vanishing Carleson measures.

To avoid compact support assumptions, we work with the unit disk $\Delta$
and the unit circle $\mathbb{T}$.

For a finite signed measure $\mu$ on $\Delta$, the \emph{balayage} of $\mu$ is defined as
\[S\mu(\zeta)=\int_{\Delta}P_{z}(\zeta)\, d\mu(z), \ \ \zeta\in \mathbb{T},\]
where $P_{z}(\zeta)$ is the Poisson kernel.

The main theorem is stated as follows:
\begin{theorem}\label{vmo}
	Let $f\in L^{1}(\mathbb{T})$. The following three statements are equivalent:
    \begin{enumerate}
        \item[$(V_1)$] $f$ lies in $\mathrm{VMO}(\mathbb{T});$
        \item[$(V_2)$] there exists a vanishing Carleson measure $\mu \in CM_0(\Delta)$ and a continuous function $g\in C(\mathbb{T})$ such that 
\[f(\zeta)=g(\zeta)+S\mu(\zeta),\ \zeta\in \mathbb{T};\]
        \item[$(V_3)$] there exists a smooth function $F\in C^{\infty}(\Delta)$ satisfying the following conditions:
\begin{enumerate}[\rm(i)]
	\item $\lim_{r\to 1}F(re^{i\theta})-f(e^{i\theta})\in C(\ttt)${\rm;}
	\item $|\nabla F|dxdy\in CM_0(\Delta)$.
\end{enumerate}
%where $C(\ttt)$ is the space of continuous functions on $\mathbb{T}$.
    \end{enumerate}
\end{theorem}
    Note that $(V_2)$ in Theorem~\ref{vmo} also allows us to reestablish the result of Korey \cite{MR1840427}, which states that every $\vmo$ function on $\mathbb{T}$ can be represented as the difference of two $\mathrm{VLO}$ functions. (See Corollary~\ref{VLO} for details.)

\medskip

The paper is organized as follows. Section~\ref{pre} introduces the necessary definitions and recalls some preliminary results. 
In Section~\ref{V1toV2}, we use an iteration argument to prove $V_1 \Rightarrow V_2$ in Theorem~\ref{vmo}. 
Section~\ref{V2toV3} establishes $V_2 \Rightarrow V_3$ by an explicit construction. Finally, Section~\ref{V3toV1} proves $V_3 \Rightarrow V_1$ and recovers a result of Korey.

\section{Preliminaries}\label{pre}

An integrable function $f$ on $\mathbb{T}$ is said to have \emph{bounded mean oscillation} (BMO), denoted by $f\in \mathrm{BMO}(\mathbb{T})$, if
\[
\|f\|_{*}\coloneqq\sup_{I}\frac{1}{|I|}\int_I |f(x)-f_{I}|\,dx < \infty,
\]
where the supremum is taken over all intervals $I\subset \mathbb{T}$, and $f_{I}=\tfrac{1}{|I|}\int_{I}f(x)dx$ is the mean value of $f$ over $I$. Note that $f_I$ can be replaced by any constant depending on $I$
(see \cite{girela2001analytic,garnett2007bounded}). A function $f$ is said to be in $\mathrm{VMO}(\mathbb{T})$ if $f\in \mathrm{BMO}(\mathbb{T})$ and 
\[
\lim_{|I|\to 0}\frac{1}{|I|}\int_I |f(x)-f_{I}|\,dx=0
\]
uniformly over intervals $I$.

Recall that a measure $\mu$ on the unit disk $\Delta$ is a \emph{Carleson measure}, denoted by $\mu\in CM(\Delta)$, if there exists a constant $C>0$ such that $\sup_{I} |\mu|(\tilde{I})/|I| < C$, where $\tilde{I}$ is the Carleson square
\[ \tilde{I}\coloneqq \{r\zeta: \zeta\in I,\  1-|I|/2\pi\leq r\leq 1\}.\]
The infimum of such constants $C$ is called the {\emph{Carleson norm} of $\mu$, denoted by $\Vert\cdot\Vert_{\mathcal{C}}$. Furthermore, if $\lim_{|I|\to 0} |\mu|(\tilde{I})/|I|=0$, we call $\mu$ is a \emph{vanishing Carleson measure}, denoted by $\mu \in CM_{0}(\Delta)$.

By Fubini's theorem,
\[\int_{\ttt}|S\mu(\zeta)||d\zeta|\leq \int_{\Delta}d|\mu|(z) \eqqcolon\left\|\mu\right\|.\]
In particular, if $\mu$ is a Carleson measure, $S\mu\in \bmo$ and the norm of $S\mu$ can be bounded by the Carleson norm of $\mu$ as follows.
\begin{theorem}[Theorem 1.6, Chapter VI, \cite{garnett2007bounded}]\label{sbmo}
	Let $\mu$ be a Carleson measure on $\Delta$. Then $S\mu\in \bmo(\ttt)$ and 
\[\left\|S\mu\right\|_{*}\leq C\left\|\mu\right\|_{\ccc}\]
for some universal constant $C$.
\end{theorem}
Conversely, every $f\in\bmo(\ttt)$ admits such a representation
up to an additive bounded function.
\begin{theorem}[\cite{carleson1976two,uchiyama1980remark} ]\label{sbmoi}
	Let $f\in \bmo(\ttt)$. There exists a Carleson measure $\mu \in CM(\Delta)$ and a function $g\in L^{\infty}(\ttt)$ such that 
\[f(\zeta)=g(\zeta)+S\mu(\zeta),\ \zeta=e^{i\theta}\in \ttt.\]
\end{theorem}

\begin{remark}\label{rmk1}
This theorem follows directly from the duality between the real Banach space $H^1$ and $\bmo$. The duality between $H^1$ and $\bmo$ can also be derived from this decomposition of $\bmo$. A constructive proof of such a decomposition was studied by Carleson \cite{carleson1976two} and by Fefferman (unpublished; see page 263 in \cite{garnett2007bounded}). The construction is described as follows: for any given increasing sequence $\{r_n\} \nearrow 1$, there exist $h_n\in L^{\infty}(r_n\ttt)$, $n=1,2,\ldots$ such that 
\[\mu(z)=\sum_{n=1}^{\infty}\delta_{r_n}(r)h_n(r_{n}e^{i\varphi}),\ z=re^{i\varphi},\]
where $\delta_{r_n}$ are the Dirac measures supported on the point $r_n$. Furthermore, we have $\left\|g\right\|_{\infty}+\sum_{n=1}^{\infty}\left\|h_{n}\right\|_{\infty}\leq C\left\|f\right\|_{*}$ (see page 152 in \cite{stein1972hp}), and the constant $C$ is independent of the function $f$.
\end{remark}

\section{Proof of $V_1 \Rightarrow V_2$}\label{V1toV2}
Since $f\in \vmo(\mathbb T)\subset\bmo(\mathbb T)$, Theorem~\ref{sbmoi}
ensures the existence of
$g_1 \in L^{\infty}(\mathbb{T})$ and a measure $\mu_{1} \in CM(\Delta)$ such that \[f(\zeta) = g_{1}(\zeta) + S\mu_{1}(\zeta).\]
By Remark~\ref{rmk1}, the measure $\mu_{1}$ admits the representation
\[
    \mu_{1}(z) = \sum_{n=1}^{\infty} \delta_{r_n}(r)\, h_n^{(1)}(r_{n}e^{i\varphi}), 
    \quad h_n^{(1)} \in L^{\infty}(r_n\mathbb{T}),\ z=re^{i\varphi}.
\]
Moreover, the uniform estimate
\[
    \|g_1\|_{\infty} + \sum_{n=1}^{\infty}\|h_{n}^{(1)}\|_{\infty} \le C\|f\|_{*}\] holds, where $C$ is a positive constant independent of $f$.

It follows from \cite[Theorem 5.1, Chapter VI]{garnett2007bounded} that there exists $r^{(1)}>0$ such that $\|f-P_{r^{(1)}}*f\|_{*}<\tfrac{\|f\|_{*}}{2}$. Using the linearity of convolution,
\[P_{r^{(1)}}*f
= P_{r^{(1)}}*g_1 + P_{r^{(1)}}*S\mu_1
\eqqcolon \tilde g_1 + S\tilde\mu_1,
\]
where \[
\tilde{\mu}_1(z)
= \sum_{n=1}^{\infty}\delta_{r_n}(r)\tilde h_n^{(1)}(r_ne^{i\varphi}), \quad\tilde h_n^{(1)} = P_{r^{(1)}}*h_n^{(1)}.\]

Since $g_1\in L^\infty(\mathbb T)$ and $h_n^{(1)}\in L^\infty(r_n\mathbb T)$,
convolution with $P_{r^{(1)}}$ yields $\tilde g_1\in C(\mathbb T)$ and
$\tilde h_n^{(1)}\in C(r_n\mathbb T)$. In addition,
\begin{align}\label{hn1}
    \|\tilde g_1\|_\infty+\sum_{n=1}^\infty\|\tilde h_n^{(1)}\|_\infty
\le
\|g_1\|_\infty+\sum_{n=1}^\infty\|h_n^{(1)}\|_\infty
\le C\|f\|_* .
\end{align}

The interchange of the order of integration below is justified by the uniform convergence of the series.

Set $f^{(1)} = f - P_{r^{(1)}} * f$, so that $ \|f^{(1)}\|_{*} \le \tfrac{1}{2} \|f\|_{*}$. We claim that $f^{(1)}\in\vmo(\mathbb T)$. 
Represent the mean integral of $f^{(1)}$ as follows,
{\small\begin{align*}
	f_{I}^{(1)}&=\frac{1}{|I|}\int_{I}f^{(1)}(e^{i\theta})\, d\theta=\frac{1}{|I|}\int_{I}\left(\int_{0}^{2\pi}\left[f(e^{i\theta})-f(e^{i(\theta-\varphi)})\right]P_{r^{(1)}}(\varphi)\, d\varphi\right)\, d\theta\\
    &=\int_{0}^{2\pi}P_{r^{(1)}}(\varphi)(f-f_{\varphi})_{I}\,d\varphi,
\end{align*}}
where $f_\varphi=f(e^{i(\theta-\varphi)})$ denotes the translation of $f(e^{i\theta})$. Then for any $I$ with sufficiently small arclength $\delta$, it follows that
{\small \begin{align*}
	&\frac{1}{|I|} \int_I\left|f^{(1)}(e^{i \theta})-f_I^{(1)}\right| \,d \theta\\
	=&\frac{1}{|I|} \int_I \left| \int_0^{2 \pi} P_{r^{(1)}}(\varphi) \left(f(e^{i \theta})-f_\varphi\right)d \varphi-\int_0^{2 \pi} P_{r^{(1)}}(\varphi)\left(f-f_{\varphi}\right)_I d \varphi\right| d\theta\\
	=& \frac{1}{|I|} \int_I\left|\int_0^{2 \pi} P_{r^{(1)}}(\varphi)\left[\left(f(e^{i \theta})-f_I\right)-\left(f_{\varphi}-\left(f_{\varphi}\right)_I\right)\right] \,d \varphi\right| \,d \theta.\\
    \leq & \int_0^{2 \pi} P_{r^{(1)}}(\varphi) \left( 
        \frac{1}{|I|} \int_I \big| f(e^{i \theta}) - f_I \big| \, d\theta 
        + \frac{1}{|I|} \int_I \big| f_{\varphi} - (f_{\varphi})_I \big| \, d\theta 
      \right) d\varphi \\
    \leq & 2 \, \sup_{|I| = \delta} \frac{1}{|I|} \int_I |f - f_I| \, d\theta.
\end{align*}}
Letting $\delta \to 0$ and using the assumption that $f \in \vmo(\mathbb{T})$, we obtain $f^{(1)}$ is in $\vmo(\mathbb{T})$. 

Next, we show that $\tilde{\mu}_{1}\in CM(\Delta)$. 
For any Carleson square $\tilde{I}$, we have
\begin{align}\label{carleson measure}
    \frac{1}{|I|} \int_{\tilde{I}} d|\tilde{\mu}_1|(z)
    &\leq \frac{1}{|I|} \int_I \sum_{\{n : r_n \ge 1 - |I|/2\pi\}}^\infty 
        \big| \tilde{h}_n^{(1)}(r_n e^{i\varphi}) \big| \, d\varphi \notag\\
    &\leq \sum_{\{n : r_n \ge 1 - |I|/2\pi\}}^\infty \| \tilde{h}_n^{(1)} \|_\infty
    \le \sum_{\{n : r_n \ge 1 - |I|/2\pi\}}^\infty \| h_n^{(1)} \|_\infty\le C\|f^{(1)}\|_*.
\end{align}
Hence, $\tilde{\mu}_{1}$ defines a Carleson measure. Therefore, we conclude that 
\[f=f^{(1)}+\tilde{g_1}+S\tilde{\mu_1}, 
\]
where $ f^{(1)}\in \vmo(\mathbb{T}),\, \tilde{g_1} \in C(\ttt)$, and $\tilde{\mu}_{1}\in CM(\Delta)$.

By applying the same argument to $f^{(1)}$, there exists $r^{(2)} > 0$ such that \[
    \| f^{(1)} - P_{r^{(2)}} * f^{(1)} \|_* 
    \leq \tfrac{\| f^{(1)} \|_*}{2} \leq \tfrac{\| f \|_*}{4}.\]

There exist a function $g_2 \in L^{\infty}(\mathbb{T})$ and a measure
\[
    \mu_2(z) = \sum_{n=1}^{\infty} \delta_{r_n}(r) \, h_n^{(2)}(r_n e^{i\varphi}), 
    \quad h_n^{(2)} \in L^{\infty}(r_n \mathbb{T}),
\]
such that
\[
    f^{(1)} = g_2 + S \mu_2,
    \quad 
    \|g_2\|_\infty + \sum_{n=1}^{\infty} \| h_n^{(2)} \|_\infty 
    \le C \| f^{(1)} \|_* \le \frac{C}{2} \| f \|_*.
\]

As before, setting $\tilde{g}_2 = P_{r^{(2)}} * g_2$ and $\tilde{\mu}_2 = P_{r^{(2)}} * \mu_2$, we obtain
\begin{align}\label{hn2}
        \|\tilde{g}_2\|_\infty + \sum_{n=1}^{\infty} \|\tilde{h}_n^{(2)}\|_\infty
    \leq \|g_2\|_\infty + \sum_{n=1}^{\infty} \| h_n^{(2)} \|_\infty 
    \leq \frac{C}{2} \| f \|_*.
\end{align}

Similarly, we let $f^{(2)} = f^{(1)} - P_{r^{(2)}} * f^{(1)}$ with $\|f^{(2)}\|_{*} \le \tfrac{1}{2} \|f^{(1)}\|_{*}$. Then one can show $f^{(1)}=f^{(2)}+\tilde{g_2}+S\tilde{\mu_2}$, where $f^{(2)}\in \vmo(\mathbb{T}),\, \tilde{g_2} \in C(\ttt)$, and $\tilde{\mu}_{2}\in CM(\Delta)$. 

Repeating the above procedure iteratively, we obtain the limit function $g$ and the limit measure $\mu$:
\[
    f = \sum_{k=1}^{\infty} \tilde{g}_k + \int_{\Delta} P_z(\zeta) \left( \sum_{k=1}^{\infty} \,d\tilde{\mu}_k(z) \right) \eqqcolon g + S\mu.
\]

Since each $\tilde{g}_k$ is continuous on $\mathbb{T}$ and
\[
    \Big\| \sum_{k=1}^{\infty} \tilde{g}_k \Big\|_\infty 
    \le \sum_{k=1}^{\infty} \| \tilde{g}_k \|_\infty 
    \le \sum_{k=1}^{\infty} \frac{C \|f\|_*}{2^{k-1}} 
    = 2 C \|f\|_*,
\]
the series $\sum_{k=1}^{\infty} \tilde{g}_k$ converges uniformly, yielding that $g \in C(\mathbb{T})$.

Finally, it remains to show that the limit measure $\mu$ is a vanishing Carleson measure. From the estimate \eqref{carleson measure}, it follows that
\begin{align*}
    \frac{1}{|I|} \int_{\tilde{I}} d|\mu|(z)
    &\le \frac{1}{|I|} \int_{\tilde{I}} \sum_{k=1}^{\infty} \,d|\tilde{\mu}_k|(z) \leq \sum_{k=1}^{\infty} \sum_{\{n : r_n \ge 1 - |I|/2\pi\}} \| \tilde{h}_n^{(k)} \|_\infty.
\end{align*}

Combined with \eqref{hn1} and \eqref{hn2}, we obtain
\[
    \sum_{k=1}^{\infty} \sum_{\{n : r_n \ge 1 - |I|/2\pi\}} \| \tilde{h}_n^{(k)} \|_\infty
    \le \sum_{k=1}^{\infty} \sum_{n=1}^{\infty} \| \tilde{h}_n^{(k)} \|_\infty
    \le \sum_{k=1}^{\infty} \frac{C \|f\|_*}{2^{k-1}} = 2 C \|f\|_*.
\]

Hence, by the dominated convergence theorem, we can interchange the limit and the series:
\[
    0 \le \lim_{|I|\to 0} \frac{1}{|I|} \int_{\tilde{I}} \,d|\mu|(z)
    \le \sum_{k=1}^{\infty} \lim_{|I|\to 0} \sum_{\{n : r_n \ge 1 - |I|/2\pi\}} \| \tilde{h}_n^{(k)} \|_\infty = 0,
\]
which shows that $\mu$ is indeed a vanishing Carleson measure.

\section{Proof of $V_2 \Rightarrow V_3$}\label{V2toV3}

In this section, we show that a $\vmo$ function differs from the boundary value of a certain smooth function $F$ satisfying the Carleson condition by a continuous function.
\medskip

Assume that $f$ admits a decomposition as in $(V_2)$. 
Our goal is to construct a function $F \in C^{\infty}(\Delta)$ satisfying the conditions stated in $(V_3)$. To this end, we introduce an auxiliary function:
\[
 \tilde{P}_z(u)\coloneqq P_z(e^{i\theta}) \chi_{(r, 1)}(\rho),
\]
where \[z=r e^{i \varphi} \in \Delta,\ u=\rho e^{i\theta} \in \Delta,\ 0<r, \rho<1,\]
and $\chi_{(r, 1)}$ denotes the characteristic function of the interval $(r, 1)$. 

Denote
\begin{align*}
	\nu_z(u)\coloneqq \frac{\partial \tilde{P}_z}{\partial \rho}(u),\quad \ \lambda_z(u)\coloneqq \frac{\partial \tilde{P}_z}{\partial \theta}(u).
\end{align*}
The measure $\nu_z$ is supported on the circle $\{u:|u|=r\}$ and coincides with the Lebesgue measure on this circle weighted by the Poisson kernel $P_z$, while $
\lambda_z(u)=\tfrac{\partial P_z(e^{i\theta})}{\partial \theta}\,
\chi_{(r,1)}(\rho)$ is a function. It follows from \cite[Lemmas 1.3.2, 1.3.3]{varopoulos1977bmo} that $\left\|\nu_z\right\|\leq 1$ and $\left\|\lambda_z\right\|\leq C_1$ for some positive constant $C_1$, where \[\left\|\nu_z\right\|\coloneqq \int_{u\in\Delta}d\nu_z(u) \qquad\text{and} \qquad\left\|\lambda_z\right\|\coloneqq  \int_{u\in\Delta}\lambda_z(u)d\rho \, d\theta.\] 

Let $\mu$ be the vanishing Carleson measure in $(V_2)$. Set
\begin{align*}
F(u) &=\int_{z\in\Delta} \tilde{P}_z(u) \,d \mu(z), \quad u=\rho e^{i\theta}\in\Delta, \\
g(e^{i\theta}) &=\int_{z\in\Delta} P_z(e^{i\theta})\, d|\mu|(z), \quad e^{i\theta} \in \ttt.
\end{align*}
Theorem \ref{sbmo} implies that $g(e^{i\theta})<+\infty$ for almost every $e^{i\theta} \in \mathbb{T}$. For such $e^{i\theta}$, applying the dominated convergence theorem yields
\[\lim_{\rho\to 1}F(\rho e^{i\theta})=\int_{\Delta}P_{z}(e^{i\theta})\, d\mu(z).\]
A direct computation gives
\begin{align}\label{df}
	\frac{\partial F}{\partial \rho}(u)=\int_{z\in\Delta}\nu_z(u)\, d\mu(z),\quad \frac{\partial F}{\partial \theta}(u)=\int_{z\in \Delta}\lambda_z(u)\, d\mu(z).
\end{align}

Note that $\tfrac{\partial F}{\partial \rho}(u)$ defines a measure which is singular with respect to the two-dimensional Lebesgue measure, whereas $\tfrac{\partial F}{\partial \theta}(u)$ appears as a density function with respect to $d\rho\, d\theta$.
\medskip

For the function $F$ defined above, we first show that $\tfrac{\partial F}{\partial \rho}(u)$ and $\tfrac{\partial F}{\partial \theta}(u) \,d\rho\, d \theta$ are vanishing Carleson measures. By \cite[Lemma 1.3.1]{varopoulos1977bmo}, these quantities are known to be Carleson measures. 

Let $h>0$ be sufficiently small, and we fix a closed interval $I_0 \subset \mathbb{T}$ with $|I_0|=h$. For each $m=0,\ldots, N_1$, let $I_m$ denote the interval concentric with $I_0$ of length $(2m+1)h$, where $N_1$ is the unique integer satisfying 
\[
    {(2N_1+1)h}< 2\pi\leq {(2N_1+3)h}.
\]
We further set $I_{N_1+1}$ as $\mathbb{T}$ and define
\[
    \hat{I}_m\coloneqq \{r\zeta: \zeta\in I_m,\  1-h/2\pi\leq r\leq 1\}.
\]
\begin{figure}[htp]
 \centering
 \includegraphics[width=6cm]{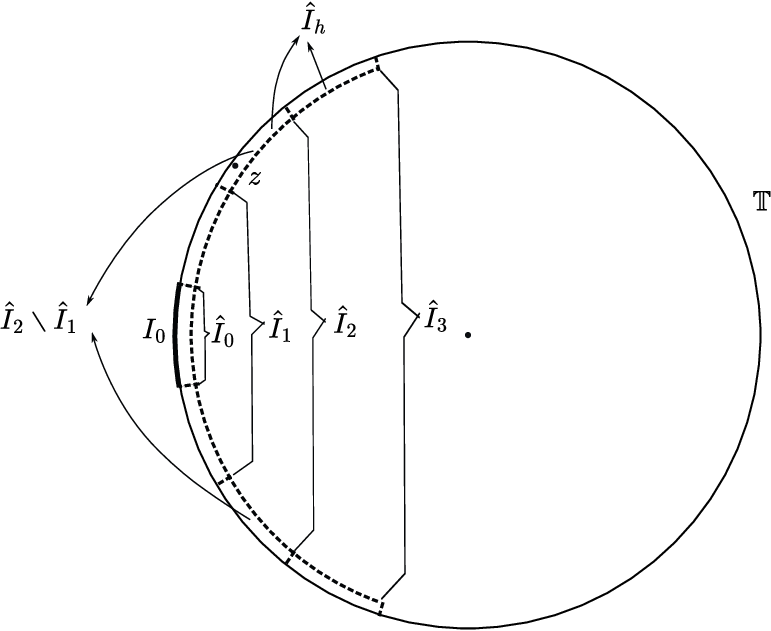}
%\captionsetup{justification=centering}
 \caption{Illustration of $\hat{I}_m$.}
 \label{hat I}
\end{figure}
Note that each $\hat{I}_m$ consists of $2m+1$ Carleson squares of side length $h$ for $m=0,\ldots, N_1$. See Fig.~\ref{hat I}. Moreover, for each $m=1,\ldots, N_1$, the difference $\hat{I}_{m}\setminus \hat{I}_{m-1}$ consists of exactly two Carleson squares of side length $h$. For simplicity, we denote any such Carleson square arising from $\hat{I}_{m}\setminus \hat{I}_{m-1}$ by $\hat{I}_h$.

A direct computation shows that there exists a universal constant $C>0$ such that
\begin{align}\label{equ:vz}
	\notag|\nu_z|(\hat{I}_0)&\leq C h \frac{1-r^2}{|z-e^{i\theta}|^2}\leq \frac{C}{m^2}, \ &\ z\in\hat{I}_{m}\setminus\hat{I}_{m-1},\ 1-r\leq h;\notag\\
	|\nu_z|(\hat{I}_0)&=0,\ &\ z=re^{i\varphi},\ 1-r> h,
\end{align}
where $m=2,\ldots,N_1+1$, \[|\nu_z|(\hat{I}_0)=\int _{u\in\hat{I}_0}d|\nu_z|(u), \quad u=\rho e^{i\theta}\in \hat{I}_0.\]

From \eqref{equ:vz} and the fact that $\|\nu_z\|\leq 1$ we conclude that
\begin{align}\label{dfdp}
\int_{u\in\hat{I}_0}\,d\left|\frac{\partial F}{\partial\rho}\right|(u) \leq&\int_{u\in\hat{I}_0}\int_{z\in \Delta}\,d|\nu_z|(u)\,d|\mu|(z)\leq\int_{z\in \Delta}|\nu_z|(\hat{I}_0)\,d|\mu|(z)\notag\\\notag
	=&\int_{\{z:|z|>1-h\}}|\nu_z|(\hat{I}_0)\,d|\mu|(z)\\
	\leq& \int_{\hat{I}_1}\left\|\nu_z\right\|\,d|\mu|(z)+\sum_{m=2}^{N_1+1}\int_{\hat{I}_{m}\setminus\hat{I}_{m-1}}\frac{C}{m^2}\,d|\mu|(z)\notag\\
	\leq& 3|\mu|(\hat{I}_h)+2C\sum_{m=1}^{\infty}\frac{|\mu|(\hat{I}_h)}{m^2}\lesssim  c_h h, 
 \end{align}
where $A\gtrsim B\, (A\lesssim B)$ means that there exists a universal constant $C$ such that $ A\geq CB \, (A\leq CB)$. Note that $\hat{I}_1$ consists of three Carleson squares of side length $h$, while for $m = 1, \ldots, N_1$, the set $\hat{I}_m \setminus \hat{I}_{m-1}$ consists of two Carleson squares of side length $h$. Moreover, $\hat{I}_{N_1+1} \setminus \hat{I}_{N_1}$ is contained in the union of two adjacent Carleson squares of side length $h$.

Since $\mu$ is a vanishing Carleson measure, it follows from \eqref{dfdp} that $\frac{\partial F}{\partial \rho}(u)$ is also a vanishing Carleson measure.
\medskip

We next consider $\tfrac{\partial F}{\partial \theta}$. 
For every $0 < h \leq 2\pi$, we recall $I_0$ is a fixed closed interval on $\mathbb{T}$ with $|I_0| = h$. Denote $N_0$ by the unique natural number such that $2^{N_0} < 2\pi \leq 2^{N_0+1}$. For $0 \leq m \leq N_0$, let $I_m$ denote the interval concentric with $I_0$ of length $2^m h$, and define $I_{N_0+1}$ as $\mathbb{T}$.
By direct computation, for $u = \rho e^{i\theta}$ we have
\begin{align}\label{dfdtheta}
\left|\tfrac{\partial F}{\partial\theta}\right|(u) \leq \int_{z\in\Delta} \left|\tfrac{\partial \tilde{P}_z(u)}{\partial \theta}\right| d|\mu|(z) 
= \int_{z \in \Delta} \left|\tfrac{\partial P_z(e^{i\theta})}{\partial \theta}\right| \chi_{(r,1)}(\rho)  d|\mu|(z),
\end{align}
and the estimate
\begin{align}\label{inequdtheta}
\left|\frac{\partial P_z(e^{i\theta})}{\partial  \theta}\right|\leq \frac{C}{|z-e^{i\theta}|^2}
\end{align}
holds for all $z=re^{i\varphi}$ and all $\theta\in [0,2\pi]$ (see\cite{varopoulos1977bmo}).

\begin{figure}[htp]
 \centering
 \includegraphics[width=6cm]{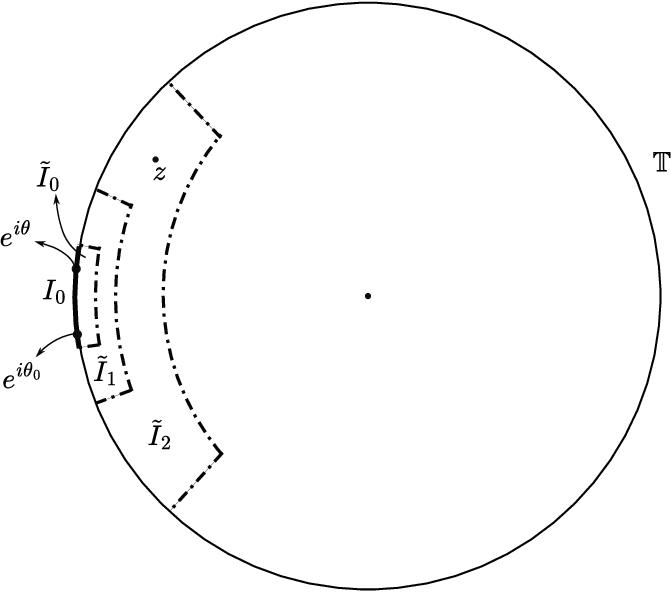}
%\captionsetup{justification=centering}
 \caption{Illustration of Carleson squares $\tilde{I}_k$.}
 \label{tilde I}
\end{figure}

Let $\tilde I_m$ be the Carleson squares defined by 
\[ \tilde{I}_m \coloneqq \{r\zeta: \zeta\in I_m,\  1-|I_m|/2\pi\leq r\leq 1\},
\] 
where $0\leq m\leq N_0+1$, see Fig~\ref{tilde I}. Let $e^{i\theta}\in I_0$, $z\in\tilde{I}_{m}\setminus\tilde{I}_{m-1}$, it follows that
\begin{align}\label{dtheta}
\left|\frac{\partial P_z(e^{i\theta})}{\partial  \theta}\right|\leq \frac{C}{|z-e^{i\theta}|^2}\lesssim \frac{1}{2^{2m}h^2}, \quad 2\leq m\leq N_0+1.
\end{align}

From \eqref{df}, \eqref{dfdtheta} and \eqref{inequdtheta}, we obtain that
\begin{align}
	\notag &\int_{u\in\tilde{I}_0}\left|\frac{\partial F}{\partial\theta}\right|(u)\,d\rho \,d \theta\notag\\
    \leq&\int_{z\in\Delta}\int_{u\in\tilde{I}_0}\left|\frac{\partial P_z(e^{i\theta})}{\partial \theta}\right|\chi_{(r,1)}(\rho)\,d\rho \,d \theta \,d|\mu|(z)\notag\\
=&\int_{z\in\tilde{I}_1}\int_{u\in\tilde{I}_0}\left|\frac{\partial P_z(e^{i\theta})}{\partial \theta}\right|\chi_{(r,1)}(\rho)\,d\rho \,d \theta\, d|\mu|(z) \notag\\
&+\sum_{m=2}^{N_0+1}\int_{z\in\tilde{I}_{m}\setminus\tilde{I}_{m-1}}\int_{u\in\tilde{I}_0}\left|\frac{\partial P_z(e^{i\theta})}{\partial \theta}\right|\chi_{(r,1)}(\rho)\,d\rho\, d \theta\, d|\mu|(z) \notag\\
	 \leq& \int_{z\in\tilde{I}_1}\left\|\lambda_z\right\|\,d|\mu|(z)+\sum_{m=2}^{N_0+1}\int_{\tilde{I}_{m}\setminus\tilde{I}_{m-1}}\int_{\tilde{I}_0}\frac{C}{|z-e^{i\theta}|^2}\,d\rho\, d \theta\, d|\mu|(z). 
\end{align}
Applying \eqref{inequdtheta}, it follows that
\begin{align*}
\int_{u\in \tilde{I}_0}\frac{1}{|z-e^{i\theta}|^2}\,d\rho \,d \theta \lesssim \frac{1}{2^{2m}}.
\end{align*}
Therefore, together with the fact that $\|\lambda_z\|\leq C_1$,
\begin{align}\label{df1}
	\int_{u\in\tilde{I}_0}\left|\frac{\partial F}{\partial\theta}\right|(u)\,d\rho\, d \theta&\lesssim C_1|\mu|(\tilde{I}_1)+ \sum_{m=2}^{N_0+1}\frac{|\mu|(\tilde{I}_m)}{2^{2m}}\notag\\
    &\lesssim C_1\left(c_{2h} h+\sum_{m=2}^{N_0+1}\frac{c_{2^mh}h}{2^m}\right).
\end{align}

It remains to show that $\sum_{m=2}^{N_0+1}\frac{c_{2^mh}}{2^{m}}\to 0$ uniformly as $h\to 0$. 

Since $\mu$ is a vanishing Carleson measure, we have $c_h\leq \|\mu\|_{\mathcal{C}}$ for all $h$ and $c_{h}\to 0$ uniformly as $h\to 0$. Let $\varepsilon>0$. Then there exists $N>0$ such that \[\sum_{m=N}^{N_0+1}\frac{c_{2^mh}}{2^{m}}\leq \sum_{m=N}^{\infty}\frac{\left\|\mu\right\|_{\ccc}}{2^m}<\varepsilon.\]

For this $\varepsilon$, we can further choose $\delta'>0$ so that whenever $2^m h<\delta'$, it follows that $c_{2^m h}<\varepsilon$. Set $\delta=\delta'2^{-N}$. Then for $m<N$ we have $2^m h<\delta'$, provided $h<\delta$. Without loss of generality, we may also assume $N<N_0+2$ (otherwise the second sum below is empty).

Now, for any arc $I_0$ with $|I_0|=h<\delta$, we obtain
\begin{align*}
	\sum_{m=2}^{N-1}\frac{c_{2^mh}}{2^m}+\sum_{m=N}^{N_0+1}\frac{c_{2^mh}}{2^m}
	\leq \ \sum_{m=2}^{\infty}\frac{\varepsilon}{2^m}+\sum_{m=N}^{\infty}\frac{\left\|\mu\right\|_{\mathcal{C}}}{2^m}< 4\varepsilon.
\end{align*}
We then conclude that $\tfrac{\partial F}{\partial \theta}(u) \,d\rho \,d \theta$ is a vanishing Carleson measure. 
\medskip

Finally, we will smooth the constructed function $F$ by truncating $P_z$ with a smooth function instead of the characteristic function:
\[\tilde{P}_z(u)=P_z(e^{i\theta}) \varphi\left(\tfrac{1-|u|}{1-|z|}\right), \ \forall\ u\in \Delta,\]
where $\varphi(t)$ is a nonnegative $C^{\infty}$ function defined as
\[
\varphi(t)=\begin{cases}0, &t\geq 2,\\
	1, &0\leq t\leq1.
\end{cases}
\]
By \eqref{dfdp} and \eqref{df1}, $|\nabla F|\,dx\, dy$ is a vanishing Carleson measure.

Indeed, by computing the gradient in polar coordinates and applying the chain rule, we have
\begin{align}\label{grad}
    |\nabla F| = \sqrt{\left(\tfrac{\partial F}{\partial \rho}\right)^2 + \left(\tfrac{1}{\rho}\tfrac{\partial F}{\partial \theta}\right)^2}.
\end{align}
For sufficiently small $h$, the factor $1/\rho$ is bounded above by $2$, so that
\[
    |\nabla F| \leq 2 \sqrt{\left(\tfrac{\partial F}{\partial \rho}\right)^2 + \left(\tfrac{\partial F}{\partial \theta}\right)^2} 
    \leq 2\left(\left|\tfrac{\partial F}{\partial \rho}\right| + \left|\tfrac{\partial F}{\partial \theta}\right|\right).
\]
Therefore, the function $F$ satisfies the required conditions, which completes the proof of $V_2 \Rightarrow V_3$.

\section{Proof of $V_3\Rightarrow V_1$}\label{V3toV1}
Assume that $F(x, y) = F(re^{i\theta}) \in C^\infty(\Delta)$ satisfies conditions $(\mathrm{i})$ and $(\mathrm{ii})$ of $(V_3)$. Since the space of continuous functions $C(\mathbb{T})$ is contained in $\mathrm{VMO}(\mathbb{T})$, it suffices to show that the boundary value of $F$, still denoted by $f$, belongs to $\mathrm{VMO}(\mathbb{T})$.

According to \cite[Theorem 1.1.2]{varopoulos1977bmo}, $f\in\bmo(\ttt)$. It remains to show that
 \[\lim _{|I| \rightarrow 0} \int_{I}\left|f(z)-f_{I}\right||d z|=0.\]
Let $I$ be any closed subinterval on $\mathbb{T}$ and consider the Carleson squares $\tilde{I}$ with $|I|=h$ ($0<h<1/10$) defined as before,
\[\tilde{I}=\{r \zeta: \zeta\in I,\ 1-h/2\pi\leq r\leq 1\}.\]
By \eqref{grad}, $|\nabla F|(x,y)$ is comparable to \[ \sqrt{\left|\tfrac{\partial F}{\partial r}\right|^2+\left|\tfrac{\partial F}{\partial \theta}\right|^2} \eqqcolon |\nabla_{r,\theta}F|(re^{i\theta}),\] so that
\[\int_{\tilde{I}}|\nabla F|(x,y)\,dx\, dy\asymp\int_{\tilde{I}}|\nabla_{r,\theta}F|(re^{i\theta})\,dr\, d\theta\leq c_h h.\]
By Fubini's theorem, there exists $r_0\in [1-h/2\pi,1]$ such that 
\begin{align}\label{inequ:ch}
	  \int_{I}|\nabla_{r,\theta}F|(r_0e^{i\theta})\, d\theta\leq c_h.
\end{align}
Otherwise, for all $1-h/2\pi\leq r\leq 1$, we have
\[\int_{I}|\nabla_{r,\theta}F|(re^{i\theta})\, d\theta> c_h,\]
which implies
\[\int_{\tilde{I}}|\nabla_{r,\theta}F|(re^{i\theta})dr\, d\theta> c_h h.\]
This contradicts the fact that $|\nabla_{r,\theta}F|(re^{i\theta})\,dr\, d\theta$ is a Carleson measure. 

Consequently, it follows from \eqref{inequ:ch} that
\begin{align}\label{11}
	|F(r_0e^{i\theta})-F(r_0e^{i\theta_0})|\lesssim c_h, \ \forall\ e^{i\theta}\in I,
\end{align}
and by direct computation,
\begin{align}\label{22}
	\left|\lim_{r\to 1}F(re^{i\theta})-F(r_0e^{i\theta})\right|\leq \int_{1-r_0}^{1}|\nabla_{r,\theta}F|(re^{i\theta})dr.
\end{align}
Combining with \eqref{11} and \eqref{22}, we conclude that
\begin{align*}
	\left|f(e^{i\theta})-F(r_0e^{i\theta_0})\right|\lesssim c_h+\int_{1-\frac{h}{2\pi}}^{1}|\nabla_{r,\theta}F|(re^{i\theta})dr,
\end{align*}
so we obtain that
\begin{align*}
	\int_{I}\left|f(e^{i\theta})-F(r_0e^{i\theta_0})\right|\, d\theta\lesssim  c_h h+\int_{\tilde{I}}|\nabla_{r,\theta}F|(re^{i\theta})dr\, d\theta\lesssim  c_h h,
\end{align*}
which implies $f\in \vmo(\ttt)$, completing the proof of Theorem \ref{vmo}.

\begin{remark}
   It is worth noting that the above argument remains valid if we weaken the assumption on $F$ from smooth to continuously differentiable.
\end{remark}
\medskip

Parallel to the study of bounded mean oscillation, the concept of \emph{bounded lower oscillation} (BLO) was also introduced. A function $f \in \mathrm{BLO}$ if $
\sup_{I}(f_{I} - \inf_{I}f) < \infty$. Coifman and Rochberg \cite{coifman1980another} proved that $\mathrm{BMO}=\mathrm{BLO}-\mathrm{BLO}$. The vanishing analogue, \emph{vanishing lower oscillation} (VLO), is defined by $\lim_{|I|\to 0}(f_{I} - \inf_{I}f)=0$.
Korey \cite{MR1840427} subsequently solved a question posed by Coifman and Rochberg \cite{coifman1980another}, proving the decomposition $\mathrm{VMO}=\mathrm{VLO}-\mathrm{VLO}$. We use $(V_2)$ to reestablish this result.

\begin{corollary}\label{VLO}
    Every $\vmo$ function on $\mathbb{T}$ can be represented as the difference of two $\mathrm{VLO}$ functions. 
\end{corollary}
\begin{proof}
    Let $f\in \mathrm{VMO}(\mathbb{T})$. Then there exists $g\in C(\mathbb{T})$ and a vanishing Carleson measure $\mu$ such that $f=g+S\mu$. By the definition of $\mathrm{VLO}$, we have $g\in \mathrm{VLO}$. Write $\mu=\mu^{+}-\mu^{-}$ as the difference of two positive measures. To obtain $\mathrm{VMO}=\mathrm{VLO}-\mathrm{VLO}$, it suffices to show that $S\mu\in\mathrm{VLO}$ whenever $\mu$ is a positive vanishing Carleson measure.

    For every $0<h\leq 2\pi$, let $I_0$ be any fixed closed interval on $\mathbb{T}$ with $|I_0|=h$. Let $I_k$ be the interval concentric with $I_0$ of length $2^kh$, where $0\leq k\leq N_0$, $N_0$ is the unique natural number such that $2^{N_0}< 2\pi\leq 2^{N_0+1}$. Define $I_{N_0+1}$ as $\mathbb{T}$.

For $e^{i\theta},e^{i\theta_0}\in I_0$ and $z\in\tilde{I}_{k}\setminus\tilde{I}_{k-1}$, $2\leq k\leq N_0+1$, where $\tilde{I}_{k}$ are Carleson squares defined before. See Fig.~\ref{tilde I}. 
A simple geometric argument (see \cite{girela2001analytic,garnett2007bounded})
shows that there exists a constant $C>0$, independent of $I_0$, such that
\begin{align}\label{pz}
	|P_{z}(e^{i\theta})-P_{z}(e^{i\theta_0})|=\left|\frac{1-|z|^2}{|e^{i\theta}-z|^2}-\frac{1-|z|^2}{|e^{i\theta_0}-z|^2}\right|\leq \frac{C}{2^{2k}|I_0|}.
\end{align}
 
We divide the unit disk into a collection of distinct Carleson squares, thereby obtaining
{\small \begin{align*}
S\mu(e^{i\theta})&=\int_{\Delta}P_{z}(e^{i\theta})\, d\mu(z)=\int_{\tilde{I}_{1}}P_{z}(e^{i\theta})\, d\mu(z)+\sum_{k=2}^{N_0+1}\int_{\tilde{I}_{k}\setminus\tilde{I}_{k-1}} P_{z}(e^{i\theta})\, d\mu(z).
\end{align*}}
Since $S\mu$ is continuous on $\mathbb{T}$, there exists
$\theta_0\in I_0$ such that
\[
S\mu(e^{i\theta_0})=\inf_{I_0} S\mu .
\]
     
By direct computation, we obtain
\begin{align*}
&\frac{1}{|I_0|}\int_{I_0}S\mu(e^{i\theta})d\theta-S\mu(e^{i\theta_0})\\	
	\leq &\frac{1}{|I_0|}\int_{I_0}\left(\int_{\widetilde{I}_1}P_{z}(e^{i\theta})d\mu(z)\right)d\theta+\frac{1}{|I_0|}\int_{I_0}\left(\int_{\widetilde{I}_1}P_{z}(e^{i\theta_0})d\mu(z)\right)d\theta\\
	&+\sum_{k=2}^{N_0+1}\frac{1}{|I_0|}\int_{I_0}\left(\int_{\widetilde{I}_{k}\backslash\widetilde{I}_{k-1}}\left|P_{z}(e^{i\theta})-P_{z}(e^{i\theta_0})\right|d\mu(z)\right)d\theta\\
    \coloneqq& A_{1,1}+A_{1,2}+A_2.
\end{align*}
By Fubini's theorem, exchanging the order of integration yields that the first two terms are bounded by $4c_{2h}$. For the third term, we invoke the estimate \eqref{pz}, which gives
\begin{align}\label{vlo_est_3}
A_2\leq \sum_{k=2}^{N_0+1}\int_{\tilde{I}_{k}\setminus\tilde{I}_{k-1}}
	\frac{C}{2^{2k}|I_0|}\, d\mu(z)
	\leq C\sum_{k=2}^{N_0+1}\frac{\mu(\tilde{I}_k)}{2^{2k}|I_0|}
	\leq C\sum_{k=2}^{N_0+1}\frac{c_{2^kh}}{2^{k}}.
\end{align}
Using the same argument as in the proof of $V_2\Rightarrow V_3$, we obtain
\[
\sum_{k=2}^{N_0+1}\frac{c_{2^kh}}{2^{k}}\to 0
\]
uniformly as $h\to 0$. Hence $S\mu\in \mathrm{VLO}$.
\end{proof}

\bibliography{vv}

@article{jones1980carleson,
  title={Carleson measures and the {F}efferman--{S}tein decomposition of $\mathrm{BMO}(\mathbb{R})$},
  author={Jones, P. W.},
  journal={Ann. Math.},
  volume={111},
  number={2},
  pages={197--208},
  year={1980},
  publisher={JSTOR},
}

@article {coifman1980another,
    AUTHOR = {Coifman, R. R. and Rochberg, R.},
     TITLE = {Another characterization of $\mathrm{BMO}$},
   JOURNAL = {Proc. Amer. Math. Soc.},
  FJOURNAL = {Proceedings of the American Mathematical Society},
    VOLUME = {79},
      YEAR = {1980},
    NUMBER = {2},
     PAGES = {249--254},
}

@incollection{MR1840427,
  author    = {Korey, M. B.},
  title     = {A decomposition of functions with vanishing mean oscillation},
  booktitle = {Harmonic Analysis and Boundary Value Problems},
  series    = {Contemp. Math.},
  volume    = {277},
  pages     = {45--59},
  publisher = {Amer. Math. Soc.},
  address   = {Providence, RI},
  year      = {2001},
}

@article {Uchiyama82Acta,
    AUTHOR = {Uchiyama, A.},
     TITLE = {A constructive proof of the {F}efferman--{S}tein decomposition
              of $\mathrm{BMO}(\mathbb{R}^n)$ },
   JOURNAL = {Acta Math.},
  FJOURNAL = {Acta Mathematica},
    VOLUME = {148},
      YEAR = {1982},
     PAGES = {215--241},
      ISSN = {0001-5962,1871-2509},
   MRCLASS = {42B30},

MRREVIEWER = {Jos\'e\ Garc\'ia-Cuerva},
       DOI = {10.1007/BF02392729},
       URL = {https://doi.org/10.1007/BF02392729},
}

@article {varopoulos1978remark,
    AUTHOR = {Varopoulos, N. T.},
     TITLE = {A remark on functions of bounded mean oscillation and bounded harmonic functions. {A}ddendum to: ``$\mathrm{BMO}$ functions and the
              {$\overline \partial $}-equation'' ({P}acific {J}. {M}ath.
              {\bf 71} (1977), no. 1, 221--273)},
   JOURNAL = {Pacific J. Math.},
  FJOURNAL = {Pacific Journal of Mathematics},
    VOLUME = {74},
      YEAR = {1978},
    NUMBER = {1},
     PAGES = {257--259},
}

@article {varopoulos1977bmo,
    AUTHOR = {Varopoulos, N. T.},
     TITLE = {$\mathrm{BMO}$ functions and the {$\overline \partial $}-equation},
   JOURNAL = {Pacific J. Math.},
  FJOURNAL = {Pacific Journal of Mathematics},
    VOLUME = {71},
      YEAR = {1977},
    NUMBER = {1},
     PAGES = {221--273},
}

@article {stein1972hp,
    AUTHOR = {Fefferman, C. and Stein, E. M.},
     TITLE = {{$H^{p}$} spaces of several variables},
   JOURNAL = {Acta Math.},
  FJOURNAL = {Acta Mathematica},
    VOLUME = {129},
      YEAR = {1972},
    NUMBER = {3--4},
     PAGES = {137--193},
}

@article {uchiyama1980remark,
    AUTHOR = {Uchiyama, A.},
     TITLE = {A remark on {C}arleson's characterization of $\mathrm{BMO}$},
   JOURNAL = {Proc. Amer. Math. Soc.},
  FJOURNAL = {Proceedings of the American Mathematical Society},
    VOLUME = {79},
      YEAR = {1980},
    NUMBER = {1},
     PAGES = {35--41},
}

@article{girela2001analytic,
  title={Analytic functions of bounded mean oscillation},
  author={Girela, D.},
  journal={Complex function spaces (Mekrij{\"a}rvi, 1999)},
  volume={4},
  pages={61--170},
  year={2001},
  publisher={Univ. Joensuu Joensuu}
}

@article {john1961functions,
    AUTHOR = {John, F. and Nirenberg, L.},
     TITLE = {On functions of bounded mean oscillation},
   JOURNAL = {Comm. Pure Appl. Math.},
  FJOURNAL = {Communications on Pure and Applied Mathematics},
    VOLUME = {14},
      YEAR = {1961},
     PAGES = {415--426},
}

@book {garnett2007bounded,
    AUTHOR = {Garnett, J. B.},
     TITLE = {Bounded analytic functions},
    SERIES = {Graduate Texts in Mathematics},
    VOLUME = {236},
 PUBLISHER = {Springer, New York},
      YEAR = {2007},
     PAGES = {xiv+459},

}

@article {carleson1976two,
    AUTHOR = {Carleson, L.},
     TITLE = {Two remarks on {$H^{1}$} and $\mathrm{BMO}$},
   JOURNAL = {Adv. Math.},
  FJOURNAL = {Advances in Mathematics},
    VOLUME = {22},
      YEAR = {1976},
    NUMBER = {3},
     PAGES = {269--277},
      ISSN = {0001-8708},
   MRCLASS = {30A78},
}

@article {sarason1975functions,
    AUTHOR = {Sarason, D.},
     TITLE = {Functions of vanishing mean oscillation},
   JOURNAL = {Trans. Amer. Math. Soc.},
  FJOURNAL = {Transactions of the American Mathematical Society},
    VOLUME = {207},
      YEAR = {1975},
     PAGES = {391--405},
}
% Text of article.

% Bibliographies can be prepared with BibTeX using amsplain,
% amsalpha, or (for "historical" overviews) natbib style.
\bibliographystyle{amsplain}
% Insert the bibliography data here.

\end{document}